\documentclass[11pt]{amsart}
\usepackage{fullpage}
\usepackage {amsmath}
\usepackage {amsthm}
\usepackage {enumerate}
\usepackage{amssymb}
\usepackage{amsfonts}
\usepackage{anysize}
\usepackage{setspace}
\usepackage[all]{xy}

\marginsize{3.15cm}{3.15cm}{1.6cm}{3cm}

\numberwithin{equation}{section}

\theoremstyle{plain}

\newtheorem{theorem}[equation]{Theorem}

\newtheorem{corollary}[equation]{Corollary}
\newtheorem{lemma}[equation]{Lemma}
\newtheorem{question}[equation]{Question}

\theoremstyle{definition}
\newtheorem{definition}[equation]{Definition}

\theoremstyle{definition}
\newtheorem{remark}[equation]{Remark}

\theoremstyle{definition}

\theoremstyle{definition}
\newtheorem{example}[equation]{Example}

\theoremstyle{definition}

\title{\scshape\bfseries $r$-skeletons on the Alexandroff duplicate}

\author{{\bfseries S. Garc\'ia-Ferreira}}
\address{Centro de Ciencias Matem\'aticas\\
         Universidad Nacional Aut\'onoma de M\'exico\\
				 Campus Morelia\\
         Apartado Postal 61-3, Santa Mar\'ia, 58089, Morelia, Michoac\'an, M\'exico.}
\email{sgarcia@matmor.unam.mx, novo1126@hotmail.com}
\author{{\bfseries C. Yescas-Aparicio}}



\date{}

\thanks{Research of the first-named author was supported
by CONACYT grant no. 176202 and PAPIIT grant no. IN-105318 }

\begin{document}

\begin{abstract} An $r$-skeleton on a compact space is a family of continuous retractions  having certain rich properties. The $r$-skeletons have  been used to characterized the Valdivia compact spaces and the Corson compact spaces. Here, we characterized  a compact space with an $r$-skeleton, for which the given $r$-skeleton can be extended to an $r$-skeleton on  the Alexandroff Duplicate of the given space. Besides, we prove that if $X$ is a zero-dimensional compact space without isolated points and $\{r_s:s\in \Gamma\}$ is an $r$-skeleton on $X$, then there is $s\in \Gamma$  such that $cl(r_s[X])$ is not countable.
\end{abstract}

\maketitle

\section{Introduction and preliminaries}

 Our spaces will be Hausdorff. The Greek letter $\omega$ will stand for the first infinite cardinal number. Given an infinite set  $X$,  the symbol $[X]^{\leq \omega}$ will denote the set of all countable  subsets of $X$ and the meaning of $[X]^{< \omega}$ should be clear.     A partially ordered set $\Gamma$ is {\it up-directed} whenever for every $s,s'\in \Gamma$, there is $t\in \Gamma$ such that $s\leq t$ and $s'\leq t$. And $\Gamma$ is $\sigma$-{\it complete} if $\sup_{n<\omega}s_n\in \Gamma$, for each increasing sequence $ \langle s_n\rangle_{n<\omega} \subseteq \Gamma$.  If $X$ is a space and $A\subseteq X$, then the closure of $A$ in $X$ will be denoted either by $\overline{A}$ or $cl_X(A)$. A space is called {\it cosmic} if it has a countable network. We recall the reader that a compact space is cosmic iff it is metrizable and separable.

 \medskip
In \cite{kubis1}, the authors introduced the notion of a {\it $r$-skeleton }, which is a family of retractions. This concept emerged as the dual notion of projectional skeletons in Banach spaces (see \cite{kubis}).

\begin{definition}\label{skeleton}
Let  $X$ be a space. An $r$-{\it skeleton} on $X$ is a family $\{r_s:s\in \Gamma\}$ of retractions in $X$, indexed by an up-directed $\sigma$-complete partially ordered set $\Gamma$, such that
\begin{enumerate}
\item[$(i)$] $r_s(X)$ is a cosmic space, for each $s\in \Gamma$;

\item[$(ii)$] $r_s=r_s\circ r_t=r_t\circ r_s$ whenever $s\leq t$;

\item[$(iii)$] if $\{s_n\}_{n<\omega}\subseteq \Gamma$, $s_n\leq s_{n+1}$ for each $n<\omega$ and $s=\sup \{s_n : n<\omega\}$, then $r_s(x)=\lim_{n\rightarrow \infty}r_{s_n}(x)$ for each $x\in X$; and

\item[$(iv)$] for each $x\in X$, $x=\lim_{s\in \Gamma}r_s(x)$.
\end{enumerate}
The {\it induced subspace} on $X$ defined by $\{r_s: s\in \Gamma\}$ is  $\bigcup_{s\in \Gamma}r_s(X)$. If $X=\bigcup_{s\in \Gamma}r_s(X)$, then we will say that $\{r_s:s\in \Gamma\}$ is a \textit{full $r$-skeleton}. Besides, if $r_s\circ r_t= r_t\circ r_s$, for any $s,t\in \Gamma$, then we will say that  $\{r_s:s\in \Gamma\}$ is \textit{commutative}.
\end{definition}

It is shown in  \cite{cuth1}  that the Corson compact spaces are those compact spaces which have a  full $r$-skeleton, and in the  paper \cite{kubis1} it is proved that the Valdivia compact spaces are  those compact spaces which   have a commutative $r$-skeleton. The ordinal space $[0,\omega_1]$ is a compact space with $r$-skeleton, but does not admit neither a commutative nor a full $r$-skeleton. Thus,   the class of spaces with $r$-skeletons contains properly the Valdivia  compact spaces. In the literature, some other authors use the name  \textit{retractional skeleton} instead of $r$-skeleton, and the name \textit{non commutative Valdivia compact space} is used to address to a  compact space with an $r$-skeleton.

\medskip

 The $r$-skeletons have been  very important  in the study of certain topological properties of compact spaces (see, for instance, \cite{cuth2}). As example of this, in the next theorem,  we will highlight the  basic properties of the induced subspace.

\begin{theorem}[Kubi\'s,\cite{kubis}]
Let $X$ be a compact space with an $r$-skeleton and $Y\subseteq X$ the  induced space by the $r$-skeleton. Then:
\begin{itemize}
\item $Y$ is  countably closed in $X$, i. e. for any $B\in [Y]^{\leq \omega}$ we have $\overline{B}\subseteq Y$,
\item  $Y$ is a Frechet-Urysohn space and
\item $\beta Y = X$.
\end{itemize}
\end{theorem}

A way to inherit the $r$-skeletons of a space to its closed subspaces  is the following.

\begin{theorem}[M. Cuth, \cite{cuth1}]\label{teocuth1}
Let $X$ be a compact space with $r$-skeleton $\{r_s:s\in \Gamma\}$ and $F$ be a closed subset of $X$. If $Y\cap F$ is dense in $F$, then $F$ admits an $r$-skeleton.
\end{theorem}

An interesting research topic has been the study of the connection between the topological properties of a space and its Alexandroff duplicate notions:

\smallskip

{\it Remember that the \textit{Alexandroff duplicate} of a space $X$, denoted by $AD(X)$,  is  the space $X \times \{0,1\}$ with the topology in which all points of $X \times \{1\}$ are isolated, and the basic neighbourhoods of the points $(x,0)$ are of the form $(U \times \{0,1\}) \setminus \{(x,1)\}$ where $U$ is a neighbourhood of $x\in X$. We denote by $X_i$ the subspace of $AD(X)$, $X\times \{i\}$. We denote by $\pi$ the projection from $AD(X)$ onto $X$. We remember that $X$ is homeomorphic to $X_0$ which, in some cases, will be identified with $X$.}

\smallskip

\noindent Following this topic, one can wonder whether an $r$-skeleton on a space can be extended to an $r$-skeleton to its Alexandroff duplicate. By using  Theorem \ref{teocuth1}, J. Somaglia   pointed out that an $r$-skeleton on an Alexandroff duplicate can  be inherited to the base space:

\begin{theorem}[\cite{soma1}]\label{soma2}
Let $X$ be a compact space. If $AD(X)$ has an $r$-skeleton, then $X$ has an $r$-skeleton.
\end{theorem}

With respect to the reciprocal implication of the previous theorem,
 Garc\'ia-Ferreira and Rojas-Hern\'andez  gave  the following partial answer.

\begin{theorem}[\cite{reynaldo1}]\label{reycorson}
If  $X$ is a compact space with full $r$-skeleton, then $AD(X)$ has a full $r$-skeleton.
\end{theorem}

 Theorem \ref{reycorson} means that the Corson property is preserved under the Alexandroff Duplicate.  J. Somaglia in \cite{soma1} showed that if
  $X$ is  a compact space with  an $r$-skeleton and $Y$ the induced space such that $X\setminus Y$ is finite, then $AD(X)$ admits an $r$-skeleton.
 This result of  Somaglia motivated this paper in which we give a necessary and sufficient conditions on the $r$-skeleton of a space $X$, in order that its Alexandroff duplicated $AD(X)$  admits an $r$-skeleton. All these  will be described in detail in the fourth  section. In the  second section, we  discuss the relation between the inverse limits and the $r$-skeletons. The third section contains results about the $r$-skeletons on zero-dimensional spaces. In particular,  we show that  if  $\{r_s:s\in \Gamma\}$ is  an $r$-skeleton on a zero-dimensional compact space $X$ without isolated points, then  $\overline{r_s[X]}$ is not countable for some $s\in \Gamma$.

\section{$r$-skeletons and Inverse Limits}

The original notion of an $r$-skeleton used  inverse limits of second countable spaces and from here one can obtain the family of retraction given in Definition \ref{skeleton} (this was done in \cite{kubis1}). In what follows, we explain in detail this connection in a different and useful manner.

\medskip

We shall consider inverse system of compact spaces whose all bonding mappings are surjetions. The symbol $\langle X_s,f^t_s,\Gamma\rangle$ will  denote a such inverse system, where $\Gamma$ is a up-directed partially ordered set and $f^t_s$ is the bonding map, for $s\leq t$. We say that $\langle X_s,f^t_s,\Gamma\rangle$ is $\sigma$-{\it complete} if  $\Gamma$ is $\sigma$-complete and $X_t=\langle X_s,f^{s'}_s,\Gamma'\rangle$, where   $\Gamma'\subseteq \Gamma$ is a countable up-directed set and  $t=\sup \Gamma'$.\\

\medskip

Next, we shall consider    $\sigma$-complete inverse  systems  $\langle X_s,f^t_s,\Gamma\rangle$ of compact spaces with the following properties:
\begin{enumerate}
\item[$(a)$] $X_s$ is a cosmic space,
\item[$(b)$] there is  $X'_s\subseteq \underleftarrow{lim}\langle X_s,f^t_s,\Gamma\rangle$ such that $P_s:=\pi_s \upharpoonright_{X'_s}:X'_s\rightarrow X_s$ is an homeomorphism, where $\pi_s$ is the projection on $X_s$, and
\item[$(c)$] if  $t \in \Gamma$ and $t\geq s$, then  $X'_s\subseteq X'_t$,
\end{enumerate}
for every  $s\in \Gamma$.

\begin{remark}\label{obvi}
  Assume that $X=\underleftarrow{lim}\langle X_s,f^t_s,\Gamma\rangle$.   For $s\in \Gamma$, we put $r_s:X\rightarrow X$ by $r_s:= P^{-1}_s\circ  \pi_s$. Observe that $x=\pi_s\circ P^{-1}_s(x)$ for each $x\in X_s$. Then, by conditions $(a)-(c)$,  we obtain that $\{r_s:s \in \Gamma\}$ is a family of retractions.
\end{remark}

\begin{lemma}\label{lemita1} Let $\langle X_s,f^t_s,\Gamma\rangle$ be a $\sigma$-complete inverse system   which satisfies  $(a)-(c)$. Then the following assertions hold:
\begin{enumerate}
\item For $t>s$ there is a continuous link function  $g^t_s:X'_t\rightarrow X'_s$  such that $r_s=g^t_s\circ r_t$,

\item $\langle X'_s,g^t_s,\Gamma\rangle$ is a $\sigma$-complete inverse system such that $\underleftarrow{lim}\langle X'_s,f^t_s,\Gamma\rangle=\{(r_s(x))_{s\in \Gamma}:x\in \underleftarrow{lim}\langle X_s,f^t_s,\Gamma\rangle\}$, and

\item $\underleftarrow{lim}\langle X'_s,f^t_s,\Gamma\rangle$ is homeomorphic to $\underleftarrow{lim}\langle X_s,f^t_s,\Gamma\rangle$.
\end{enumerate}
\end{lemma}
\begin{proof}
We set $X=\underleftarrow{lim}\langle X_s,f^t_s,\Gamma\rangle$.\\
$(1)$ For $s\leq t$, we  define $g^t_s:=P^{-1}_s\circ f^t_s\circ P_t$. By   definition, $g^t_s$ is a continuos function and
$$
\begin{array}{ll}
g^t_s\circ r_t&=P^{-1}_s\circ f^t_s \circ P_t\circ r_t\\
&=P^{-1}_s\circ f^t_s \circ P_t\circ P^{-1}_t\circ \pi_t\\
&=P^{-1}_s\circ f^t_s \circ \pi_t\\
&=P^{-1}_s\circ \pi_s\\
&=r_s.
\end{array}
$$
Hence, $g^t_s$ is the desired function.\\
$(2)$ If $s\leq t\leq t'$, then  we have that
$$
\begin{array}{ll}
g^t_s\circ g^{t'}_t&=P^{-1}_s\circ f^t_s \circ P_t\circ P^{-1}_t\circ f^{t'}_t \circ P_{t'}\\
&=P^{-1}_s\circ f^t_s \circ f^{t'}_t \circ P_{t'}\\
&=P^{-1}_s\circ f^{t'}_s \circ P_{t'}\\
&=g^{t'}_s.
\end{array}
$$
 This shows that $\langle X'_s,g^t_s,\Gamma\rangle$ is an inverse system. Since  that $X'_s$ is homeomorphic to $X_s$ and the $\sigma$-completeness of $\langle X_s,f^t_s,\Gamma\rangle$, we deduce the $\sigma$-completeness of $\langle X'_s,g^t_s,\Gamma\rangle$.
It is not hard to see that the mapping
$h:X \to \underleftarrow{lim}\langle X'_s,g^t_s,\Gamma\rangle
$ defined by $h((x_s)_{s\in \Gamma})=(r_s(x))_{s\in \Gamma}$ is continuous. Hence, $h(X)=\{(r_s(x))_{s\in \Gamma}:x\in X\}\subseteq \underleftarrow{lim}\langle X'_s,g^t_s,\Gamma\rangle$.  We shall prove that $h(X)$ is dense in $\underleftarrow{lim}\langle X'_s,g^t_s,\Gamma\rangle$. Indeed, let $z\in \underleftarrow{lim}\langle X'_s,g^t_s,\Gamma\rangle$ and $U=\bigcap^n_{i=1}\pi^{-1}_{t_i}(U_i)$, where $U_i\subseteq X_{t_i}$ is an open set such that $\pi_{t_i}(z)\in U_i$, for $1\leq i\leq n$. Let us take $t \in \Gamma$ such that $t\geq t_i$ for all $1 \leq i \leq n$, and $y\in X_t$ so that $r_t(y)=\pi_t(z)$. We observe that
$$
\begin{array}{ll}
\pi_{t_i}(z)&=g^t_{t_i}(\pi_t(z))\\
&=g^t_{t_i}(r_t(y))\\
&=g^t_{t_i}\circ r_t(y)\\
&=r_{t_i}(y),
\end{array}
$$
for every $1\leq i\leq n$. Hence, we have that $(r_s(y))_{s\in \Gamma}\in U$.  Therefore, $h(X)$ is dense in  $\underleftarrow{lim}\langle X'_s,g^t_s,\Gamma\rangle$. By the compactness of $h(X)$, we have $h(X)=\underleftarrow{lim}\langle X'_s,g^t_s,\Gamma\rangle$.\\

$(3)$ Since $h$ is a biyective continuous function between compact spaces, we obtain that $X$ is homeomorphic to $\underleftarrow{lim}\langle X'_s,g^t_s,\Gamma\rangle$.
\end{proof}

\begin{lemma}\label{lemita}
 Let $\langle X_s,f^t_s,\Gamma\rangle$ be a $\sigma$-complete inverse system    which satisfies  $(a)-(c)$.   If  $\gamma>s$ and  $x\in \underleftarrow{lim}\langle X_s,f^t_s,\Gamma\rangle$, then $\pi_s(r_\gamma(x))=\pi_s(x)$.
\end{lemma}
\begin{proof}
The equality  easily  follows from the relations
$$
\begin{array}{ll}
\pi_{s}(r_\gamma(x))&=\pi_{s}(P^{-1}_\gamma\circ \pi_\gamma(x))\\
&=f^\gamma_{s}\circ\pi_\gamma\circ P^{-1}_\gamma\circ\pi_\gamma(x)\\
&=f^\gamma_{s}\circ\pi_\gamma(x)\\
&=\pi_{s}(x).
\end{array}
$$
\end{proof}

In the next theorem, we describe the relation between the compact spaces with $r$-skeleton and the inverse systems.

\begin{theorem}
Let  $X$  be a compact space. Then the following conditions are equivalent:
\begin{enumerate}
\item  $X=\underleftarrow{lim}\langle X_s,f^t_s,\Gamma\rangle$, where  $\langle X_s,f^t_s,\Gamma\rangle$ is a $\sigma$-complete inverse system   which satisfies  $(a)-(c)$.
\item $X$ admits an $r$-skeleton.
\end{enumerate}

\end{theorem}
\begin{proof}
$(1)\Rightarrow (2)$. According to clause $(3)$ of Lemma \ref{lemita1}, we can identified  $X$ with $\underleftarrow{lim}\langle X'_s,f^t_s,\Gamma\rangle$.  We shall prove that the family of retractions $\{r_s:s\in \Gamma\}$ from the Remark \ref{obvi} satisfy the conditions  $(i)-(iv)$ of Definition \ref{skeleton}. Condition $(i)$ follows from  $(a)$. Now, let  $s,t\in \Gamma$ with $s\leq t$. On one hand, we have
$$
\begin{array}{ll}
r_s \circ r_t &=P^{-1}_s\circ \pi_s \circ P^{-1}_t\circ \pi_t\\
&=P^{-1}_s\circ f^t_s\circ\pi_t \circ P^{-1}_t\circ \pi_t\\
&=P^{-1}_s\circ f^t_s\circ \pi_t\\
&=P^{-1}_s\circ\pi_s=r_s.
\end{array}
$$
 On the other hand, since $X'_s\subseteq X'_t$, we have  $r_s(x)\in X'_t$, for all $x\in X$.  We note that $r_t$ fixes the elements of  $X'_t$, hence  $r_t\circ r_s(x)=r_s(x)$ for all $x\in X$. So, $r_s=r_s\circ r_t=r_t\circ r_s$ and then condition $(ii)$ holds.
 For the condition $(iii)$, let $\langle s_n\rangle_{n< \omega}$ be an increasing sequence   in $\Gamma$ and $s=\sup_{n<\omega} s_n$. We shall prove that
 $r_s(x)=\lim_{n\rightarrow \infty} r_{s_n}(x)$, inside of  $X'_s$, for every $x\in X$. Fix $x\in X$. By the $\sigma$-completeness  of the inverse system,
 $X'_s=\underleftarrow{lim}\langle X'_{s_n},g^{s_m}_{s_n},\langle s_n\rangle_{n< \omega}\rangle$.
 We pick a basic open set  $U=\bigcap^k_{i=1}\pi^{-1}_{s_{n_i}}(U_i)$ in $X'_s$  which contains  $r_s(x)$, where $U_i$ is open in $X'_{s_{n_i}}$ for each $1\leq i \leq k$.  Using  condition $(ii)$  and the fact that   $s \in \Gamma$ is an upper bound of $\{s_{n_1},...,s_{n_k}\}$, we have $ \pi_{s_{n_i}}(r_s(x))=r_{s_{n_i}}(r_s(x))=r_{s_{n_i}}(x)$, for all $1\leq i \leq k$. Since $r_s(x)\in U$,  $r_{s_{n_i}}(x)\in U_i$, for all $i$.  Now, we pick $n'<\omega$ such that $s_{n'}$ is an upper bound of $\{s_{n_1},...,s_{n_k}\}$. If $n\geq n'$, then we have that $\pi_{s_{n_i}}(r_{s_n}(x))=r_{s_{n_i}}(r_{s_n}(x))=r_{s_{n_i}}(x)\in U_i$ for all $1 \leq i \leq k$. Thus, we obtain  that $r_{s_n}(x)\in U$ for each $n \geq n'$. Therefore, $r_s(x)=\lim_{n\rightarrow \infty} r_{s_n}(x)$, for all $x\in X$.
 Finally, we shall prove  $(iv)$; that is, $x=\lim_{s\in \Gamma}r_s(x)$, for all $x\in X$. Fix $x \in X$ and let $U=\bigcap^k_{i=1}\pi^{-1}_{s_i}(U_i)$ a basic open set with $x\in U$, where $U_i$ is open in $X_{s_i}$ for every $ 1 \leq i \leq k$.   Observe that  $x_{s_i}=\pi_{s_i}(x)\in U_i$  for every $ 1 \leq i \leq k$. Pick an upper bound $s' \in \Gamma$ of $\{s_1,...,s_k\}$. If $s>s'$, by  Lemma \ref{lemita}, we then have that $\pi_{s_i}(r_s(x))=\pi_{s_i}(x)\in U_i$. Therefore, $r_s(x)\in U$ for every $s\geq s'$.
Thus, we have shown  that $\{r_s:s\in \Gamma\}$ is an $r$-skeleton on $X$.\\

$(2)\Rightarrow (1)$ Let $\{r_s:s \in \Gamma\}$ be an $r$-skeleton on $X$. We know from the paper \cite{kubis1} that $X=\underleftarrow{lim}\langle X_s,R^t_s,\Gamma\rangle$, where $\langle X_s,R^t_s,\Gamma\rangle$ is a $\sigma$-complete inverse system, $X_s=r_s(X)$  and $R^t_s:=r_s\upharpoonright_{X_t}$. The conditions $(a)-(c)$ are easy to verify.
\end{proof}

\section{$r$-skeletons on zero-dimensional spaces}

It is very natural to ask when an $r$-skeleton $\{r_s:s\in \Gamma\}$ satisfies that $|r_s[X]| \leq \omega$ for each $s\in \Gamma$. In the next results, we give a necessary condition to have this property. In this section, we will use usual terminology of trees.    By $\{0,1\}$ we denote the discrete space with two elements. We will denote the Cantor space by $\mathcal{C}$, $\{0,1\}^{<\omega}$  is the set of all finite sequences of $\{0,1\}$; $\sigma\hat{\hspace{0.1cm}}\lambda$  will be denote the usual  concatenation   for $\sigma, \lambda\in \{0,1\}^{<\omega}$,  and  the initial segment of length $n$ will denoted by  $p\mid n$ for $p\in \mathcal{C}$. For terminology not mentioned and used in the next results, we can consult \cite{kech}.

\begin{lemma}\label{lema2}

Let $X$ be a zero-dimensional compact space without isolated points. If $\{r_s:s\in \Gamma\}$ is an $r$-skeleton on $X$, then there is a Cantor scheme $\{U_\sigma:\sigma\in \{0,1\}^{<\omega}\}$ and $\{x_\sigma:\sigma\in \{0,1\}^{<\omega}\}\subset\bigcup_{s\in\Gamma}r_s(X)$ such that:
\begin{enumerate}
\item $U_\sigma$ is a clopen set,
\item $x_\sigma\in U_\sigma$ and for $i\in \{0,1\}$, $x_{\sigma\hat{\hspace{0.1cm}}i}\notin U_{\sigma\hat{\hspace{0.1cm}}i}$.
\end{enumerate}
Moreover, for $p\in \mathcal{C}$, $\bigcap_{n<\omega}U_{p\mid n}\neq\emptyset$.
\end{lemma}

\begin{proof}
For induction on length of $\sigma$. We define $U_\emptyset:=X$ and choose a point $x_\emptyset\in Y$. Now, let's suppose that  $U_\sigma$ and $x_\sigma$ are defined for all $\sigma\in \{0,1\}^{<\omega}$ of length $n$. Take $\sigma\in \{0,1\}^{<\omega}$ and suppose that length of $\sigma$ is $n$. Since $x_\sigma$ is not isolated, then there are $z, z'\in U_\sigma$ not equals to $x_\sigma$. Pick two disjointed clopen subsets  $U$ and $V$ which $z\in U$, $z'\in V$, $x_\sigma\notin U$ and $x_\sigma\notin V$. Let be $U_{\sigma\hat{\hspace{0.1cm}}0}:= U_\sigma\cap U$ and   $U_{\sigma\hat{\hspace{0.1cm}}1}:= U_\sigma\cap V$. Using the density of $Y$, we can choose $x_{\sigma\hat{\hspace{0.1cm}}0}\in U_{\sigma\hat{\hspace{0.1cm}}0}\cap Y$ and $x_{\sigma\hat{\hspace{0.1cm}}1}\in U_{\sigma\hat{\hspace{0.1cm}}1}\cap Y$. Now, pick $p\in \mathcal{C}$.  For construction,  $\{U_{p\mid n}:n< \omega\}$ has the finite intersection property. By compacity, it follows that $\bigcap_{n<\omega}U_{p\mid n}\neq\emptyset$.
\end{proof}

\begin{theorem}\label{isol}
Let $X$ be a zero-dimensional compact space without isolated points. If $\{r_s:s\in \Gamma\}$ is an $r$-skeleton on $X$, then there is $s\in \Gamma$  such that $\overline{r_s[X]}$ is not countable.
\end{theorem}

\begin{proof}
Let us consider $\{U_\sigma:\sigma\in \{0,1\}^{<\omega}\}$ and $\{x_\sigma:\sigma \in \{0,1\}^{<\omega}\}\subset\bigcup_{s\in\Gamma}r_s(X)$ like the Lemma \ref{lema2}.
We define  $F=\{x_\sigma:\sigma\in \{0,1\}^{<\omega}\}$. Since $F\in [Y]^{\leq \omega}$ and $Y$ is countable closed, then $\overline{F}\subseteq Y$. We shall prove that $\overline{F}$ is not countable. For $p\in \mathcal{C}$, let be $V_p=\bigcap_{n<\omega}U_{p\mid n}$ and $x_p\in \overline{\{x_{p\mid n}:n<\omega\}}$.
We observe that $x_p\notin\{x_{p\mid n}:n<\omega\}$. Since $\{x_{p\mid n}:n<\omega\}\in [Y]^{\leq \omega}$ and $Y$ is countable closed, we have $x_p\in Y$.
Also, $x_p\in V_p$. Let's suppose $x_p\notin V_p$. So, there is $n_0<\omega$ such that $x_p\notin U_{p\mid n_0}$.
By construction, $\{x_{p\mid m}:m\geq n_0\}\subset U_{p\mid n_0}$. Since that $U_{p\mid n_0}$ is clopen, then $X\setminus U_{p\mid n_0}$ is clopen.
 We note that $\{x_{p\mid m}:m<n_0\}\subset X\setminus U_{p\mid n_0}$.
Since $X$ is Hausdorff, we can find an open set $U$ such that $x_p\in U\subset X\setminus U_{p\mid n_0}$ and $U\cap \{x_{p\mid m}:m<n_0\}=\emptyset$. The last sentence is a contradiction because $x_p \in \overline{\{x_{p\mid n}:n<\omega\}}$. Finally, if $p,q\in \mathcal{C}$ with $p\neq q$, using that $V_p\cap V_q=\emptyset$ we have $x_p\neq x_q$. Since that $\{x_p:p\in \mathcal{C}\}\subseteq \overline{F}$, we conclude that $\overline{F}$ is not countable.
\end{proof}

\begin{corollary}\label{corisol}
 Let $X$ be a compact space. If $\{r_s:s\in \Gamma\}$ is an $r$-skeleton on $X$ such that $|r_s[X]| \leq \omega$ for all $s\in \Gamma$, then $X$ has a dense subset consisting of isolated points.
\end{corollary}

\begin{proof}
Let's suppose $U\subseteq X$ be a open subset without isolated points. We take a open subset of $X$ such that $\overline{W}\subseteq U$. We have that $Y\cap \overline{W}$ is dense in $\overline{W}$, by the Theorem \ref{teocuth1} $\overline{W}$ admits an $r$-skeleton. Moreover, the $r$-skeleton is a subfamily of restrictions on $\overline{W}$ of the family $\{r_s:s\in \Gamma\}$. Using the Proposition \ref{isol}, we have a contradiction.
\end{proof}

\begin{example}
Let $\alpha$ be a cardinal number with $\alpha \geq \omega_1$. The space $[0,\alpha]$ with the $r$-skeleton given by Kubi\'s and Michalewski in \cite{kubis1}, is an example that satisfied  the Corollary \ref{corisol}. Besides, Somaglia in \cite{soma1} proves that  the double circle of Alexandroff of $[0,\omega_2]$ also admits such $r$-skeleton.
\end{example}

\begin{example}
In the paper \cite{reynaldo1}, the authors proves that the Alexandroff duplicate of a Corson space its again a Corson space.  The Alexandroff duplicate of a Corson space   contains a dense set of isolated points and admits an $r$-skeleton, but it could  fail  that $|r_s[AD(X)]| \leq \omega$ for every $s\in \Gamma$, this shows that condition of Corollary \ref{corisol} is not sufficient.
\end{example}

\begin{question}
Which compact spaces that contain a dense set of isolated points  admit an $r$-skeletons with countable images?
\end{question}


\section{$r$-skeletons and the Alexandroff Duplicate}


In this section, we study some topological  properties of the Alexandroff Duplicate of a compact space with an $r$-skeleton. In particular, we give conditions to extend an $r$-skeleton on $X$ to an $r$-skeleton on  the Alexandroff duplicate $AD(X)$.

\begin{lemma}[\cite{reynaldo1}]\label{rey1}
Let $X$ be a set and $\Gamma$ be an up-directed partially ordered set. If for $x\in X$ there is an assignment $s_x\in \Gamma$, then  there is a function $\psi:[X]^{\leq \omega}\rightarrow \Gamma$ such that:
\begin{enumerate}
\item For $x\in X$, $\psi(x)\geq s_x$;
\item if $A\subseteq B\in [X]^{\leq \omega}$, then $\psi(A)\leq \psi(B)$; and
\item if $\langle A_n\rangle_{n<\omega}\subseteq [X]^{\leq \omega}$ with $A_n\subseteq A_{n+1}$ and $A=\bigcup_{n<\omega}A_n$, then $\psi(A)=\sup \{\psi(A_n):n<\omega\}$.
\end{enumerate}
\end{lemma}

\begin{lemma}[\cite{casa1}]\label{rey2}
Let $X$ be a compact space and $\{r_s:s\in \Gamma\}$ be a family of retractions on  $X$ which satisfied the conditions $(i)-(iii)$ of the $r$-skeleton definition. Then for $x\in \overline{\bigcup_{s\in \Gamma}r_s(X)}$, $x=\lim_{s\in \Gamma}r_s(x)$.
\end{lemma}

The next result appears implicitly in the development of several articles on the subject, here we provide a proof of it.
\begin{theorem}\label{numerables}
Let $X$ be a compact space, $\{r_s:s\in \Gamma\}$ be an $r$-skeleton on $X$ and $Y$ be the induced space. Then there is an $r$-skeleton $\{R_A: A\in [Y]^{\leq \omega}\}$ on $X$ such that
\begin{itemize}
\item $A\subseteq R_A(X)$, for all $A\in [Y]^{\leq \omega}$, and
\item $Y=\bigcup_{A\in [Y]^{\leq \omega}} R_A(X)$.
\end{itemize}
\end{theorem}
\begin{proof}
For each $x\in Y$, we pick $s_x\in \Gamma$ such that $r_{s_x}(x)=x$.  We consider the function $\psi:[Y]^{\leq \omega}\rightarrow \Gamma$  given by Lemma \ref{rey1}, which satisfied $(1)-(3)$.
Now, for each  $A\in [Y]^{\leq \omega}$ we define $R_A:X\rightarrow X$ by $R_A:=r_{\psi(A)}$.
 We shall prove that family $\{R_A:A\in [Y]^{\leq \omega}\}$ is an $r$-skeleton on $X$.  Since $\psi$ satisfy (1)-(3) of Lemma \ref{rey1} and $\{r_s:s\in \Gamma\}$ is an $r$-skeleton, trivially  $\{R_A:A\in [Y]^{\leq \omega}\}$ satisfy $(i)-(iii)$ of the $r$-skeleton definition.
 Finally, by Lemma \ref{rey2},   we have that $x=\lim_{A\in [Y]^{\leq \omega}}R_A(x)$ for all $x\in \overline{\bigcup_{A\in [Y]^{\leq \omega} }R_A(X)}$. By the choice of the family $\{s_x:x\in Y\}$, we trivially obtain that $A\subseteq R_A(X)$, for all $A\in [Y]^{\leq \omega}$. It then follows that $\bigcup_{A\in [Y]^{\leq \omega}}R_A(X)=Y$. Therefore, $\{R_A: A\in [Y]^{\leq \omega}\}$ is an $r$-skeleton on $X$.
\end{proof}

\begin{lemma}\label{teo1}
Let $X$ be a compact space, suppose that $AD(X)$ has an $r$-skeleton  with induced space $\hat{Y}$ and  $Y=\pi(\hat{Y}\cap X_0)$. Then for any $B\in [X\setminus Y]^{\leq \omega}$ we have
\begin{itemize}
\item[$(*)$]  $B$ is discrete in $X \setminus Y$, and
\item[$(**)$]  $cl_X(B)\setminus B\subseteq Y$ and $cl_X(B)\setminus B$  is a cosmic space.
\end{itemize}
\end{lemma}
\begin{proof}
By the Theorem \ref{teo1}, we may considerer an $r$-skeleton on $AD(X)$ of the form $\{r_C:C\in [\hat{Y}]^{\leq \omega}\}$. Let $B\in [X\setminus Y]^{\leq \omega}$. If $B$ is a finite set, then  the result follows immediately. Let us suppose that $|B|=\omega$ and $p\in cl_X(B)\setminus B$. Observe that $(p,0)\in cl_{AD(X)}(B\times\{1\})$. Since $B\times\{1\}\subseteq\hat{Y}$ and $\hat{Y}$ is countably closed, it follows that $(p,0)\in \hat{Y}$ and so $p\in Y$. Hence, $cl_X(B)\setminus B\subseteq Y$ and  we deduce that $B$ is discrete in $X\setminus Y$.  Now, we note that $(cl_X(B)\setminus B)\times \{0\}\subseteq cl_{AD(X)}(B\times\{1\})$. Since that $B\times \{1\}\in [\hat{Y}]^{\leq \omega}$, we have $cl_{AD(X)}(B\times\{1\})\subseteq r_{B\times \{1\}}(X)$.  Using the cosmicity of $r_{B\times \{1\}}(X)$, we conclude that $cl_X(B)\setminus B\subseteq Y$ is also cosmic.
\end{proof}

\begin{lemma}\label{lemma51}
Let $X$ be a compact space which admits an $r$-skeleton  $\{r_s:s\in \Gamma\}$ with induced space $Y$. Suppose that   for   $B\in [X\setminus Y]^{\leq \omega}$ the conditions $(*)$ and $(**)$ from above hold. Then for each $s\in \Gamma$ such that $cl_X(B)\setminus B\subseteq r_s(X)$ and for every $A\in [r_s(X)]^{\leq\omega}$, the mapping  $R_{(A,B,s)}:AD(X)\rightarrow AD(X)$ defined as
\[
R_{(A,B,s)}(x,i) :=
\begin{cases}
(x,1) & \textrm{ if $x \in A\cup B$ and $i=1$}\\
(r_{s}(x),0) & \textrm{ in another case}, \\
\end{cases}
\]
 for every $(x,i)\in AD(X)$, is a retraction on $AD(X)$.
\end{lemma}

\begin{proof}
Let $A\in [r_s(X)]^{\leq\omega}$ and assume that $cl_X(B)\setminus B\subseteq r_s(X)$. First, we prove that   $R_{(A,B,s)}$ is  a continuous function. Let $\langle(x_\lambda,i_\lambda) \rangle_{\lambda\in \Lambda} $ be  net such that $(x_\lambda,i_\lambda)\rightarrow (x,i)$. If $i=1$, then $\langle (x_\lambda,i_\lambda) \rangle_{\lambda\in \Lambda} $ is eventually constant and hence  $\langle R_{(A,B)}(x_\lambda,i_\lambda)\rangle_{\lambda\in \Lambda} $  is so.  Now, we consider the case when $i=0$ and assume that $\langle(x_\lambda,i_\lambda) \rangle_{\lambda\in \Lambda} $ is not trivial. Hence,  $R_{(A,B,s)}(x,0)=(r_s(x),0)$. We may suppose that either $\langle (x_\lambda,i_\lambda) \rangle_{\lambda\in \Lambda}\subseteq X_0$ or $\langle (x_\lambda,i_\lambda) \rangle_{\lambda\in \Lambda}\subseteq X_1$. As $\langle (x_\lambda,i_\lambda) \rangle_{\lambda\in \Lambda}\subseteq X_0$, we have that  $R_{(A,B,s)}(x_\lambda,i_\lambda)=(r_s(x),0)$ for all $\lambda \in \Lambda$, and since $r_s$ continuous, we obtain that $ R_{(A,B,s)}(x_\lambda,i_\lambda)\rightarrow R_{(A,B,s)}(x,0)$. Now,  assume that $\langle (x_\lambda,i_\lambda) \rangle_{\lambda\in \Lambda}\subseteq X_1$. If $\langle x_\lambda \rangle_{\lambda\in \Lambda}$ contains a subnet that lies eventually  in $A\cup B$, then $x\in cl_X(A\cup B)$.  Without loss of generality, we suppose that $x_\lambda \in A \cup B$ for every $\lambda \in \Lambda$. Then, $R_{(A,B,s)}(x_\lambda,i_\lambda)=(x_\lambda,1)$ for all $\lambda\in \Lambda$.  If $x\in B$,  then we deduce from  $(*)$ that $\langle x_\lambda\rangle_{\lambda \in \Lambda}$ is eventually in $A$. Since $A\subseteq r_s(X)$, we have $x\in r_s(X)$,  but $Y\cap B=\emptyset$. So, $x\in cl_X(A)\cup \big( cl_X(B)\setminus B\big)$. Since $A\cup \big( cl_X(B)\setminus B\big) \subseteq r_s(X)$, we deduce $r_s(x)=x$ and we conclude  $R_{(A,B,s)}(x_\lambda,i_\lambda)\rightarrow R_{(A,B,s)}(x,0)$. In other hand, $\langle x_\lambda \rangle_{\lambda\in \Lambda}\subseteq X\setminus \big(A\cup B\big)$ and $R_{(A,B,s)}(x_\lambda,i_\lambda) = (r_{s}(x_\lambda),0)$. Using the continuity of $r_{s}$, we have that $R_{(A,B,s)} (x_\lambda,i_\lambda)=(r_s(x_\lambda),0)\rightarrow  (r_s(x),0)=R_{(A,B,s)} (x,0)$. Therefore, $R_{(A,B,s)}$ is  a continuous function. Finally, if $(x,i)\in AD(X)$, then
\begin{align*}
R_{(A,B,s)}\circ R_{(A,B,s)}(x,i) &=
\begin{cases}
R_{(A,B,s)}(x,1) & \textrm{ if $x \in A\cup B$ and $i=1$}\\
R_{(A,B,s)}(r_{s}(x),0) & \textrm{ in another case}, \\
\end{cases}\\
&=
\begin{cases}
(x,1) & \textrm{ if $x \in A\cup B$ and $i=1$}\\
(r_{s}(r_{s}(x)),0) & \textrm{ in another case}, \\
\end{cases}\\
&=
\begin{cases}
(x,1) & \textrm{ if $x \in A\cup B$ and $i=1$}\\
(r_{s}(x),0) & \textrm{ in another case}, \\
\end{cases}\\
&=R_{(A,B,s)}(x,i).
\end{align*}
That is, $R_{(A,B,s)}$ is a retract on $AD(X)$.
\end{proof}

The main theorem of  this article is the following. In the next, we will consider    $\sigma$-complete up-directed partially ordered  sets $\Gamma$, where $\Gamma$ will be a subset of  $[Y]^{\leq \omega}\times[X\setminus Y]^{\leq \omega}$ with the order $\preceq$ defined by $(A,B)\preceq (A',B')$ if $A\subseteq A'$ and $B\subseteq B'$. Also, we remark that sometimes if $\langle(A_n,B_n)\rangle_{ n<\omega}\subseteq \Gamma$, then $\sup_\Gamma\langle(A_n,B_n)\rangle_{ n<\omega}$ is not necessarily $(\bigcup_{n<\omega}A_n,\bigcup_{n<\omega}B_n)$.

\medskip

The following remark  plays a very important rule in the proof of the next theorem.

\begin{remark}\label{obscuth}
Let $X$ be a compact space with $r$-skeleton $\{r_s:s\in \Gamma\}$ and $F$ be a closed subset of $X$. The Theorem \ref{teocuth1} asserts that   if $Y\cap F$ is dense in $F$, then $F$ admits an $r$-skeleton. In the proof of  this theorem, M. Cuth proved that the subfamily $\Gamma'=\{s\in \Gamma:r_s(F)\subseteq F\}$ is $\sigma$-closed and cofinal in $\Gamma$ and that  $\{r_s\upharpoonright_{F}:s\in \Gamma'\}$ is the $r$-skeleton on $F$. These facts will be very useful in the proof of the next theorem.
\end{remark}

\begin{theorem}\label{principal}
Let be  $X$ a compact space.
$AD(X)$ admits an $r$-skeleton if and only if  there is an $r$-skeleton $\{r_{(A,B)}:(A,B)\in \Gamma\}$ on $X$ with induced space $Y$ such that $\Gamma\subseteq [Y]^{\leq \omega}\times [X\setminus Y]^{\leq \omega}$  is $\sigma$-closed and cofinal in $[Y]^{\leq \omega}\times [X\setminus Y]^{\leq \omega}$ and  the next conditions hold:

\noindent  For every $B\in [X\setminus Y]^{\leq \omega}$,
\begin{itemize}
\item[$(*)$]  $B$ is discrete in $X \setminus Y$,
\item[$(**)$]  $cl_X(B)\setminus B\subseteq Y$ and $cl_X(B)\setminus B$ is cosmic;
\end{itemize}
 and
\begin{itemize}
\item[$(***)$]  $A\subseteq r_{(A,B)}(X)$ and  $cl_X(B)\setminus B\subseteq r_{(A,B)}(X)$ for every $(A,B)\in \Gamma$.
\end{itemize}
\end{theorem}
\begin{proof}
Necessity. Suppose that $AD(X)$ admits an $r$-skeleton with induced space $\hat{Y}$ and observe that $X_0\subseteq \hat{Y}$. According to Theorem \ref{numerables}, we may assume that such $r$-skeleton  is of the form  $\{R_C:C\in [\hat{Y}]^{\leq \omega}\}$ and satisfies the properties of the theorem. Consider the set $Y=\pi(\hat{Y}\cap X_0)$  and put $\Gamma'=[Y]^{\leq \omega}\times [X\setminus Y]^{\leq \omega}$.  For $(A,B)\in \Gamma'$, we define $R_{(A,B)}=R_{A\times \{0,1\}\cup B\times\{1\}}$. We claim that $\{R_{(A,B)}:(A,B)\in \Gamma'\}$ is an $r$-skeleton on $AD(X)$ with induced space $\hat{Y}$. The conditions $(i), (ii)$ and $(iii)$ are hold because $\{R_C:C\in [\hat{Y}]^{\leq \omega}\}$ is an $r$-skeleton. For the condition $(iv)$, let $(x,i)\in \hat{Y}$. Choose $(A,B)\in \Gamma'$ so that $(x,i)\in A\times \{0,1\}\cup B\times\{1\}$. It follows that $(x,i)=R_{A\times \{0,1\}\cup B\times\{1\}}(x,i)=R_{(A,B)}(x,i)$.
Thus, we have proved that  $\{R_{(A,B)}:(A,B)\in \Gamma'\}$ is an $r$-skeleton on $AD(X)$ with induced space $\hat{Y}$. By Remark \ref{obscuth},
 the set $\Gamma=\{(A,B)\in \Gamma':R_{(A,B)}(X_0)\subseteq X_0\}$ is $\sigma$-closed and cofinal in $\Gamma'$ and $\{R_{(A,B)}\upharpoonright_{X_0}:(A,B)\in \Gamma\}$ is an $r$-skeleton on $X_0$, with induced  space  $\hat{Y}\cap X_0$. For each $(A,B)\in \Gamma$, we define $r_{(A,B)}=\pi(R_{(A,B)}\upharpoonright_{X_0})$. Hence,  $\{r_{(A,B)}:(A,B)\in \Gamma\}$ is an $r$-skeleton on $X$ with induced space $Y$. By Lemma \ref{teo1}, the conditions $(*)$ and $(**)$  hold.  The  condition $(***)$ is easy to verify.\\

Sufficiency. Now, let $\{r_{(A,B)}:(A,B)\in \Gamma\}$ an $r$-skeleton on $X$ which satisfies the condition $(*)-(***)$, where $\Gamma\subseteq [Y]^{\leq \omega}\times [X\setminus Y]^{\leq \omega}\}$  is $\sigma$-closed and cofinal in $[Y]^{\leq \omega}\times [X\setminus Y]^{\leq \omega}$ and $Y$ the induced space.
For each $(A,B)\in \Gamma$, we know that $cl_X(B)\setminus B\subseteq r_{(A,B)}(X)$ and $A\subseteq r_{(A,B)}(X)$. Set $R_{(A,B)}=R_{(A,B,(A,B))}$, where  $R_{(A,B,(A,B))}$  is the retraction of  Lemma \ref{lemma51}.
 We claim that $\{R_{(A,B)}:(A,B)\in \Gamma\}$ is an $r$-skeleton on $AD(X)$. Indeed, we shall prove that the conditions $(i)-(iv)$ of the $r$-skeleton definition  hold.
\begin{enumerate}
\item[$(i)$] If  $(A,B)$, then $R_{(A,B)}(AD(X))=(r_{(A,B)}(X)\times \{0\})\cup ((A\cup B) \times \{1\})$ is a cosmic space.
\item[$(ii)$] Let  $(A,B)\preceq (A',B')$. Fix $(x,i)\in AD(X)$. Then
\begin{align*}
R_{(A,B)}\circ R_{(A',B')}(x,i)&=
\begin{cases}
R_{(A,B)}(x,1) & \mbox{ if }x \in A'\cup B'\mbox{ and }i=1\\
R_{(A,B)}(r_{(A',B')}(x),0) & \mbox{ in another case} \\
\end{cases} \\
&
=
\begin{cases}
\begin{cases}
(x,1) & \mbox{ if }x \in A\cup B\mbox{ and }i=1\\
(r_{(A,B)}(x),0) & \mbox{ if }x \in (A'\cup B')\setminus (A\cup B) \mbox{ and }i=1
\end{cases}
\\
(r_{(A,B)}(r_{(A',B')}(x)),0)  \mbox{ in another case} \\
\end{cases} \\
&=
\begin{cases}
(x,1) & \mbox{ if }x \in A\cup B\mbox{ and }i=1\\
(r_{(A,B)}(x),0) & \mbox{ in another case} \\
\end{cases} \\
&= R_{(A,B)}(x,i).
\end{align*}

And we also have that
\begin{align*}
R_{(A',B')}\circ R_{(A,B)}(x,i)&=
\begin{cases}
R_{(A',B')}(x,1) & \mbox{ if }x \in A\cup B\mbox{ and }i=1\\
R_{(A',B')}(r_{(A,B)}(x),0) & \mbox{ in another case} \\
\end{cases} \\
&=
\begin{cases}
(x,1) & \mbox{ if }x \in A\cup B\mbox{ and }i=1\\
(r_{(A',B')}(r_{(A,B)}(x)),0) & \mbox{ in another case} \\
\end{cases} \\
&=
\begin{cases}
(x,1) & \mbox{ if }x \in A\cup B\mbox{ and }i=1\\
(r_{(A,B)}(x),0) & \mbox{ in another case} \\
\end{cases} \\
&= R_{(A,B)}(x,i).
\end{align*}
 Therefore, $R_{(A,B)}=R_{(A,B)}\circ R_{(A',B')}=R_{(A',B')}\circ R_{(A,B)}$ whenever  $(A,B)\preceq (A',B')$.

\item[$(iii)$] Let $\langle(A_n,B_n)\rangle_{ n<\omega}\subseteq \Gamma$ be such that $(A_n,B_n)\preceq(A_{n+1},B_{n+1})$.  Let us observe that $\sup\{ (A_n,B_n): n<\omega \}=(\sup_\Gamma\{A_n:n<\omega\},\sup_\Gamma\{B_n:n<\omega\})$. For simplify put  $A=\sup_\Gamma\{A_n:n<\omega\}$ and $B=\sup_\Gamma\{B_n:n<\omega\}$. Fix $(x,i)\in AD(X)$. We will prove that $R_{(A,B)}(x,i)=\lim_{n\rightarrow \infty}R_{(A_n,B_n)}(x,i)$. In fact, if $i=0$, then $R_{(A,B)}(x,0)=(r_{(A,B)}(x),0)$ and $R_{(A_n,B_n)}(x,0)=(r_{(A_n,B_n)}(x),0)$ for all $n < \omega$. Since  $r_{(A,B)}(x)=\lim_{n\rightarrow\infty}r_{(A_n,B_n)}(x)$, we conclude that $R_{(A,B)}(x,0)=  \lim_{n\rightarrow \infty}R_{(A_n,B_n)}(x,0)$.
Now, we consider the case when $i=1$. If $x\in A\cup B$, then there is $n_0<\omega$ such that  $x\in A_n\cup B_n$ for all $n\geq n_0$.
 Hence, $R_{(A,B)}(x,1)=(x,1)=R_{(A_n,B_n)}(x,1)$, for $n\geq n_0$.  It follows that $R_{(A,B)}(x,1)=  \lim_{n\rightarrow \infty}R_{(A_n,B_n)}(x,1)$.
  If $x\notin A\cup B$, then $R_{(A,B)}(x,1)=  (r_{(A,B)}(x),0)$ and $R_{(A_n,B_n)}(x,1)=(r_{(A_n,B_n)}(x),0)$, for every $n < \omega$.
  Since the equality $r_{(A,B)}(x)=\lim_{n\rightarrow\infty}r_{(A_n,B_n)}(x)$ holds,
$R_{(A,B)}(x,1)=  \lim_{n\rightarrow \infty}R_{(A_n,B_n)}(x,1).$

\item[$(iv)$] Let $(x,i)\in AD(X)$. First, we notice that the equality $x=\lim_{(A,B)\in \Gamma}r_{(A,B)}(x)$ implies that $(x,0)=\lim_{(A,B)\in \Gamma}R_{(A,B)}(x,0)$. Now let  $x\in Y$. By cofinality of $\Gamma$, there is $(A,B)\in \Gamma$ such that $\{x\}\subseteq A$. Using condition $(***)$, we have $(x,1)=(r_{(A,B)}(x),1)=R_{(A,B)}(x,1)$. Now, we suppose $x\in X\setminus Y$. By using  the confinality of $\Gamma$, there is $(A,B)\in \Gamma$ with $\{x\}\subseteq B$.  It follows, $(x,1)=R_{(A,B)}(x,1)$.  With all, $(x,i)=\lim_{(A,B)\in \Gamma}R_{(A,B)}(x,i)$.
\end{enumerate}

\end{proof}
 From the proof of the before theorem, we can deduce the next corollary.
\begin{corollary}\label{principalvaldivia}
   Let be  $X$ a compact space such that $AD(X)$ admits a commutative $r$-skeleton.  Then   the $r$-skeleton   $\{r_{(A,B)}:(A,B)\in \Gamma\}$ on $X$ obtained in the Theorem \ref{principal} is commutative.
\end{corollary}

It is well known that there are  Valdivia compact spaces wich Duplicate  Alexandroff is not Valdivia compact. If we have a commutative $r$-skeleton   $\{r_{(A,B)}:(A,B)\in \Gamma\}$ on $X$,  the conditions given in the Theorem  \ref{principal} are not  clear for extend  to a commutative $r$-skeleton on $AD(X)$. In the next result, we add one more condition for we can to extend commutative $r$-skeletons.

\begin{corollary}
Let be  $X$ a compact space and  $\{r_{(A,B)}:(A,B)\in \Gamma\}$ a commutative $r$-skeleton  on $X$ as in the Theorem \ref{principal} wich satisfied
\begin{itemize}
\item[$(****)$] for every $(A,B),(A',B')\in \Gamma$, $r_{(A,B)}(x)=r_{(A',B')}(r_{(A,B)}(x))$, for each $x\in B'\setminus B$.
\end{itemize}
Then $AD(X)$ admits a commutative $r$-skeleton.
\end{corollary}

The next example is another application of Theorem \ref{principal}.
\begin{corollary}
Let $X$ be a compact space. If $AD(X)$ admits an $r$-skeleton, then  the induced space $Y=\pi(\hat{Y}\cap X_0)$ of $X$ is unique. That is, if $Y'$ is an induced set by an arbitrary $r$-skeleton on $X$, then $Y'=Y$.
\end{corollary}

\begin{proof}
Let $Y'$ be a subset of $X$ induced by an $r$-skeleton on $X$. We suppose that $Y\neq Y'$.  Since $Y\neq Y'$,  by Lemma 3.2 from \cite{cuth1}, we have that  $Y\cap Y'$ cannot be dense in $X$. Hence, there is a nonempty  open subset $V$ of $X$ such that $V\cap (Y\cap Y')=\emptyset$. Let $W$ a nonempty open subset such that $cl_X(W)\subseteq V$. By density of $Y'$, there is $B\subseteq W \cap Y'$ infinite and countable. Since $Y'$  is countably closed,  we must have $cl_{X}(B)\subseteq cl_X(W)\cap Y'\subseteq V\cap Y'$. On the other hand,  Theorem \ref{principal} implies that $cl_{X}(B)\setminus B\subseteq Y$, but this is impossible since    $V\cap (Y\cap Y')=\emptyset$. Therefore, $Y=Y'$.
\end{proof}

As a consequence of the previous corollary,  if $AD(X)$ has an $r$-skeleton, then $X$ is not  a super Valdivia space (for the definition of super Valdivia see the paper \cite{kalenda1}). In particular, Alexandroff duplicate  of $[0,1]^{\kappa}$ does not admit  an $r$-skeleton, for every $\kappa\geq \omega_1$.

\bigskip
Remember that for a up-directed $\sigma$-complete partially ordered set $\Gamma$ and a set $Y$, a function $\psi:\Gamma\rightarrow [Y]^{\leq \omega}$ is called monotone provided that:
\begin{itemize}
\item if $s, t\in \Gamma$, then $\psi(s)\subseteq \psi(t)$; and
\item if $\{s_n:n<\omega\}\subseteq \Gamma$ with $s_n\leq s_{n+1}$ for each $n<\omega$, then $\psi(\sup_{n<\omega}s_n)=\bigcap_{n<\omega}\psi(s_n)$.
\end{itemize}   Hence, in terms of $\omega$-monotonous functions we have the next result.

\begin{theorem}\label{novo2}
Let $X$ be a compact space which admits an $r$-skeleton with induced space $Y$. Let us suppose that  $\{r_A:A\in [Y]^{\leq \omega}\}$ is the $r$-skeleton obtained by Theorem \ref{numerables}. If the conditions $(*)$ and $(**)$ from above hold,  and  there is $\psi:[X\setminus Y]^{\leq\omega}\rightarrow [Y]^{\leq \omega}$ $\omega$-monotonous such that for $B\in [X\setminus Y]^{\leq \omega}$, $cl_X(B)\setminus B \subseteq r_{\psi(B)}(X)$. Then $AD(X)$ has an $r$-skeleton.
\end{theorem}

\begin{proof}
Let $[Y]^{\leq \omega}\times[X\setminus Y]^{\leq \omega}$. For each $(A,B)\in \Gamma'$, let $r_{(A,B)}:X\rightarrow X$ the function defined  by $r_{(A,B)}=r_{A\cup \psi(B)}$. The family $\{r_{(A,B)}:(A,B)\in \Gamma'\}$ is an $r$-skeleton on $X$ wich satisfied the conditions $(*)-(***)$. Therefore, using Theorem \ref{principal}, $AD(X)$ has an $r$-skeleton.
\end{proof}

The next is an example of aplication of Theorem \ref{principal}.
\begin{example}
Let $\kappa$ an infinite cardinal number. We considerer the $r$-skeleton given  in the article \cite{kubis1}, $\{r'_A:A\in \mathcal{A}\}$, where $\mathcal{A}$ is the  collection of all closed countable sets of $[0,\kappa]$ such that if $A\in \mathcal{A}$, then $0\in A$ and every isolated point of $A$ is isolated in $[0,\kappa]$. And the retractions are defined by $r'_A(x)=\max\{y\in A:y\leq x\}$, for each $A\in \mathcal{A}$. For this $r$-skeleton, $Y=\bigcup\mathcal{A}$ is the induced space. By propierties of the ordinal space $[0,\kappa]$ it follows that $(*)$ and $(**)$ are hold. By Theorem \ref{numerables}, from $\{r'_A:A\in \mathcal{A}\}$ we obtain an $r$-skeleton $\{r_A:A\in [Y]^{\leq \omega}\}$ with  induced space $Y$ that satisfies $(*)$ and $(**)$.   Now, if  $\psi:[X\setminus Y]^{\leq \omega}\rightarrow [Y]^{\leq \omega}$ is the function defined by $\psi(B)=\{\beta+1:\beta\in B\}$, for each $B\in [X\setminus Y]^{\leq \omega}$,  then $\psi$ is  $\omega$-monotonous. We observe that for $B\in [X\setminus Y]^{\leq \omega}$, $cl_X(B)\setminus B\subseteq  cl_X(\psi(B))\subseteq r_{\psi(B)}([0,\kappa])$. As a consequence of Theorem \ref{novo2}, we have that $AD([0,\kappa])$ has an $r$-skeleton.
\end{example}

 We do not know whether or not the monotony of  Theorem \ref{novo2} is sufficient.


\end{document}